\documentclass[12pt]{article}

\usepackage{amsrefs}
\usepackage{amsmath,amssymb,amsthm,enumerate}
\usepackage{tikz,graphicx}
\usepackage[all,cmtip]{xy}
\usepackage[applemac]{inputenc}

\newcommand \al{\alpha}
\newcommand \bs{\backslash}
\newcommand \C{{\mathbb C}}

\newcommand \CF{\mathcal F}
\newcommand \CN{\mathcal N}
\newcommand \CO{\mathcal O}

\newcommand \e{\emph}

\newcommand \fin{\mathrm{fin}}
\newcommand \g{\mathfrak g}
\newcommand \Ga{\Gamma}
\newcommand \ga{\gamma}
\newcommand \GL{\operatorname{GL}}

\newcommand \Hom{\operatorname{Hom}}
\newcommand \la{\lambda}
\newcommand \La{\Lambda}

\newcommand \M{\operatorname M}

\newcommand \N{{\mathbb N}}
\newcommand \ol{\overline}
\newcommand \om{\omega}
\newcommand \Q{{\mathbb Q}}
\newcommand \R{{\mathbb R}}
\newcommand \SL{\operatorname{SL}}
\newcommand \sm{\smallsetminus}
\newcommand \Spec{\operatorname{Spec}}
\newcommand \supp{\operatorname{supp}}

\newcommand \tr{\operatorname{tr}}
\newcommand \vol{\operatorname{vol}}
\newcommand \z{\mathfrak z}

\renewcommand \1{{\bf 1}}
\renewcommand \Im{\operatorname{Im}}
\renewcommand \({\left(}
\renewcommand \){\right)}

\renewcommand{\sp}[1]{\left\langle #1\right\rangle}
\newcommand{\norm}[1]{\left\| {#1}\right\|}

\newtheorem{theorem}{Theorem}[section]

\newtheorem{lemma}[theorem]{Lemma}

\newtheorem{proposition}[theorem]{Proposition}
\theoremstyle{definition}
\newtheorem{remark}[theorem]{Remark}

\newtheorem{definition}[theorem]{Definition}

\begin{document}

\pagestyle{myheadings} \markright{SPECTRAL THEORY FOR NON-UNITARY TWISTS}

\title{Spectral theory for non-unitary twists}
\author{Anton Deitmar}
\date{}
\maketitle

{\bf Abstract:} 
Let $G$ be a Lie-group and $\Ga\subset G$ a cocompact lattice. For a finite-dimensional, not necessarily unitary representation $\om$ of $\Ga$ we show that the $G$-representation on $L^2(\Ga\bs G,\om)$ admits a complete filtration with irreducible quotients.
As a consequence, we show the trace formula for non-unitary twists and arbitrary locally compact groups.

$$ $$

\tableofcontents

\newpage
\section*{Introduction}
For unitary representations of locally compact groups there is a general spectral theory, expressing such representations as direct integrals of irreducibles or even, if the representation is, say, trace class, as direct sums of irreducibles.
For non-unitary representations there is no spectral theory in general.
In this paper we introduce a spectral theory for representations, non-unitarily induced from cocompact lattices.
These representations occur naturally in extensions of the trace formula \cite{mueller}.
In this paper we use the spectral analysis of the group Laplacian to deduce that for a Lie group these representations admit complete filtrations with irreducible graded steps.

\section{Trace class representations}

Let $G$ be a locally compact group.
For the convenience of the reader we briefly recall the definition of the space $C_c^\infty(G)$ of test functions on $G$.

\begin{definition}
First, if $L$ is a Lie group, then $C_c^\infty(L)$ is defined as the space of all infinitely differentiable functions of compact support on $L$.
The space $C_c^\infty(L)$ is the inductive limit of all $C_K^\infty(L)$, where $K\subset L$ runs through all compact subsets of $L$ and $C_K^\infty(L)$ is the space of all smooth functions supported inside $K$. The latter is a Fr\'echet space equipped with the supremum norms over all derivatives. Then $C_c^\infty(L)$ is equipped with the inductive limit topology in the category of locally convex spaces as defined in \cite{Schaef}, Chap II, Sec. 6.

Next, suppose the locally compact group $H$ has the property that $H/H^0$ is compact, where $H^0$ is the connected component.
Let $\CN$ be the family of all normal closed subgroups $N\subset H$ such that $H/N$ is a Lie group with finitely many connected components.
We call $H/N$ a \e{Lie quotient} of $H$.
Then, by \cite{MZ}, the set $\CN$ is directed by inverse inclusion and 
$$
H\cong \lim_{\substack{\leftarrow\\ N}}H/N,
$$
where the inverse limit runs over the set $\CN$.
So $H$ is a projective limit of Lie groups.
The space $C_c^\infty(H)$ is then defined to be the sum of all spaces $C_c^\infty(H/N)$ as $N$ varies in $\CN$.
Then $C_c^\infty(H)$ is the inductive limit over all $C_c^\infty(L)$ running over all Lie quotients $L$ of $H$ and so $C_c^\infty(H)$ again is equipped with the inductive limit topology in the category of locally convex spaces.

Finally to the general case.
By \cite{MZ} one knows that every locally compact group $G$ has an open subgroup $H$ such that $H/H^0$ is compact, so $H$ is a projective limit of connected Lie groups in a canonical way.
A Lie quotient of $H$ then is called a \e{local Lie quotient} of $G$.
We have the notion $C_c^\infty(H)$ and for any $g\in G$ we define $C_c^\infty(gH)$ to be the set of functions $f$ on the coset $gH$ such that $x\mapsto f(gx)$ lies in $C_c^\infty(H)$.
We then define $C_c^\infty(G)$ to be the sum of all $C_c^\infty(gH)$, where $g$ varies in $G$.
Then $C_c^\infty(G)$  is the inductive limit over all finite sums of the spaces $C_c^\infty(gH)$.
Note that the definition is independent of the choice of $H$, since, given a second open group $H'$, the support of any given $f\in C_c(G)$ will only meet finitely many left cosets $gH''$ of the open subgroup $H''=H\cap H'$.
It follows in particular, that $C_c^\infty(G)$ is the inductive limit over a family of Fr\'echet spaces.
This concludes the definition of the space $C_c^\infty(G)$ of test functions.

\begin{remark}
\begin{enumerate}[\rm (a)]
\item Note that the inductive limit topology in the category of locally convex spaces differs from the inductive limit topology in the category of topological spaces, as is made clear in \cite{Gloeck}.
\item Note that for a linear functional $\al:C_c^\infty(G)\to\C$ to be continuous, it suffices, that for any 
local Lie quotient $L$ of $G$ and any compact subset $K\subset L$ and any sequence $f_n\in C_K^\infty(L)$ with $f_n\to 0$ in the Fr\'echet space $C_K^\infty(L)$ and every $g\in G$ the sequence $\al(L_gf_n)$ tends to zero, where $L_gf(x)=f(g^{-1}x)$.
This is deduced from \cite{Schaef}, Chap II, Sec. 6.1.
\item If a locally compact group $G$ is a projective limit of Lie groups $G_j=G/N_j$, then it follows from \cite{Neeb}*{Cor. 12.3}, that every irreducible continuous representation $\pi$ factors through some $G_j$. This reduces many issues related to distribution characters to the case of Lie groups.
\end{enumerate}
\end{remark}
\end{definition} 

\begin{definition}
A representation $(\pi,V_\pi)$ of a locally compact group $G$ is called a \e{compact representation}, if $\pi(f)$ is a compact  operator for every $f\in C_c^\infty(G)$.
It is called a \e{trace class representation}, if $\pi(f)$ is trace class  for every $f\in C_c^\infty(G)$.
We say that $G$ is a \e{trace class group}, if every irreducible unitary representation is trace class.
See \cite{DvD} for more on trace class groups.
\end{definition}

\begin{definition}
Let $G$ be a unimodular locally compact group and let $\Ga\subset G$ be a discrete subgroup. Then there exists a non-vanishing, $G$-invariant Radon measure on the quotient $\Ga\bs G$, which is unique up to scaling \cite{HA2} and the induced representation $R_g\phi(x)=\phi(xg)$ of $G$ on the Hilbert space $L^2(\Ga\bs G)$ is unitary. 
\end{definition}

Note that if $G$ admits a cocompact discrete subgroup $\Ga$, then $G$ is unimodular and $\Ga$ is a lattice.

\begin{proposition}
For a unimodular locally compact group $G$ and a discrete subgroup $\Ga\subset G$ the following are equivalent:
\begin{enumerate}[\rm (a)]
\item $\Ga\bs G$ is compact.
\item The representation of $G$ on $L^2(\Ga\bs G)$ is trace class.
\item The representation of $G$ on $L^2(\Ga\bs G)$ is compact.
\end{enumerate}
\end{proposition}

\begin{proof} 
(a)$\Rightarrow$(b)
is the classical trace formula argument and can be found in \cite{HA2}, Chapter 9.

(b)$\Rightarrow$(c) is trivial.

(c)$\Rightarrow$(a): Assume that $L^2(\Ga\bs G)$ is a compact representation, but $\Ga\bs G$ is not compact.
Then for every compact unit-neighborhood $U\subset G$ and every compact set $K\subset G$ there exists $x\in G$ such that $\Ga xU\cap \Ga K=\emptyset$, for otherwise the element $\Ga x$ of $\Ga\bs G$ lies in the compact set $\Ga\bs\Ga KU^{-1}$.

Applying this iteratedly, one obtains a sequence $x_1,x_2,\dots\in G$ such that 
$$
\Ga x_iU^2\cap \Ga x_j U^2=\emptyset
$$
for all $i\ne j$.
Fix a symmetric unit-neighborhood $V$ such that $V^3\subset U$ and let 
$$
\phi_j=c_j\1_{\Ga x_jV^2},
$$
where $c_j>0$ is such that $\norm{\phi_j}^2=\int_{\Ga\bs G}|\phi_j(x)|^2\,dx=1$.
Fix some $f\in C_c^\infty(G)$ with support in $V$ and such that $f\ge 0$ and $\int_Gf(x)\,dx=1$.
Now $\supp(R(f)\phi_j)\subset \Ga x_j V^3\subset \Ga x_j U$ and therefore the supports of $R(f)\phi_j$ for varying $j$ are disjoint, hence these vectors in the Hilbert space $L^2(\Ga\bs G)$ are pairwise orthogonal.
As $R(f)$ is a compact operator, the sequence $(R(f)\phi_j)$ must have a convergent subsequence, but as the vectors are pairwise orthogonal, there must exist a subsequence  with $\norm{R(f)\phi_{j_k}}\to 0$.
Since the integral of $f$ is 1, we have
$$
c_j\1_{\Ga x_j V}\le R(f)\phi_j.
$$
Now assume that $V^2\subset \bigcup_{l=1}^nVz_l$ with $z_l\in G$, then $\Ga x_j V^2\subset \bigcup_{l=1}^n \Ga x_j Vz_l$ and so
$$
1=\norm{\phi_j}^2=c_j^2\vol(\Ga\bs \Ga x_jV^2)\le nc_j^2\vol(\Ga\bs\Ga x_j V)
$$
and
$$
\norm{R(f)\phi_j}^2\ge \norm{c_j\1_{\Ga x_jV}}^2\ge\frac{c_j^2}n\vol(\Ga x_j V^2)=\frac1n.
$$
So there is no subsequence with $\norm{R(f)\phi_{j_k}}\to 0$.
\end{proof}

\begin{remark} (Counterexample)
In \cite{DvD}, last Remark of Section 4, it is asked, whether any locally compact group $G$ admitting a cocompact lattice must be trace class.
We now give a counterexample. Let $G=\M_2(\R)\rtimes\SL_2(\R)$, where $\M_2(\R)$ is the space of real $2\times 2$ matrices.
Let $D$ be a quaternion division algebra over $\Q$ which splits at infinity.
Fix a splitting $D\hookrightarrow \M_2(\R)$ and thus consider $D$ a $\Q$-subalgebra of $\M_2(\R)$. Fix an order $\CO\subset D$ (see \cite{Reiner}), and let $\CO^1\subset\CO^\times$ be the subgroup of all elements of determinant 1. Set
$$
\Ga=\CO\rtimes \CO^1.
$$
Since $\CO$ is a cocompact lattice in $\M_2(\R)$ and $\CO^1$ is a cocompact lattice in $\SL_2(\R)$, the group $\Ga$ is a cocompact lattice in $G$.

Next we need to show that $G$ is not trace class. In \cite{DvD}, Proposition 1.9, it is shown that $H=\R^2\rtimes\SL_2(\R)$ is not trace class. 
Let $N\subset G$ be the set of all elements of the form $(A,1)$, where the matrix $A$ has zeros in the first column. Then $N$ is closed and normal in $G$ and $G/N\cong H$.
So the irreducible representation of $H$, which is not trace class, induces an irreducible representation of $G$, which is not trace class.
\end{remark}

\begin{remark}
It seems to be an open question whether for any lattice $\Ga$ the spectral multiplicities in the discrete spectrum of $L^2(\Ga\bs G)$ are finite.
In other words, let $\Ga$ be a lattice in the locally compact group $G$ and let $(\pi,V_\pi)\in\widehat G$ be an irreducible unitary representation, is it true that
$$
\Hom_G(V_\pi,L^2(\Ga\bs G)
$$
is finite-dimensional?
\end{remark}

\section{The spectral filtration}
By a \e{representation} we shall mean a continuous representation on a Banach space.

\begin{definition}
Let $L$ be a linearly ordered set.
For $a<b$ in $L$ we consider the closed interval $[a,b]$ of all $x\in L$ with $a\le x\le b$.
The elements $a<b$ are called \e{neighbored}, if $[a,b]=\{a,b\}$, i.e., if there is no element between them.

A linearly ordered set $C$ is called \e{complete}, if every subset of $C$ possesses a supremum and an infimum.

For a given linearly ordered set $L$ there is a uniquely defined \e{completion} $C$, which is a complete ordered set which contains $L$ as a substructure such that $L$ is dense in $C$ in the sense that every $c\in C$ is the supremum or the infimum of a subset of $L$.
\end{definition}

\begin{definition}
A \e{sub-tower} is a linearly ordered set $L$, such that every $x\in L$ has a neighbor. 
A \e{tower} is a linearly ordered set which is the completion of a sub-tower.
In particular, a tower $L$ contains a minimum $\min(L)$ and a maximum $\max(L)$.
\end{definition}

\begin{definition}
Let $L$ be a tower.
Let $(R,V)$ be a representation of a locally compact group $G$ on a Banach space $V$.
A \e{complete $L$-filtration} on $(R,V)$ is a family of closed, $G$-stable subspaces $(F_i)_{i\in L}$ such that the following hold:
\begin{enumerate}[\quad\rm (a)]
\item $F_{\min(L)}=0$ and $F_{\max(L)}=V$,
\item if $i\le j$, then $F_i\subset F_j$, 
\item if $i<j$ are neighbored, then $F_j/F_i$ is irreducible, 
\item if $b\in L$ has no lower neighbor, then $F_b$ is the closure of $\bigcup_{j<b}F_j$,
\item if $a$ has no upper neighbor, then $F_a=\bigcap_{j>a}F_j$.
\end{enumerate}
\end{definition}

\begin{definition}
Let $(F_j)_{j\in L}$ be a complete $L$-filtration of $(R,V)$ for the tower $L$.
For a given irreducible representation $(\pi,V_\pi)$ of $G$ we define the \e{multiplicity} $m_L(\pi)$ to be the number of pairs $i<j$ in $L$ such that the representation on $F_j/F_i$ is isomorphic to $\pi$. This multiplicity may be zero, a natural number, or infinity.
\end{definition}

\begin{definition}
Let $(R,V)$ be a representation. A \e{subquotient} of $(R,V)$ is a representation of the form $P/Q$, where $Q\subset P$ are closed, $G$-stable subspaces of $V$.

A representation $(R,V)$ is called \e{discrete}, if every subquotient has an irreducible subquotient.
This means that for any two closed, $G$-stable subspaces $Q\subset P$ there exist closed, $G$-stable subspaces $Q\subset U\subset W\subset P$ such that $W/U$ is irreducible.
\end{definition}

\begin{lemma}
Let $(R,V)$ be a discrete representation of the locally compact group $G$.
\begin{enumerate}[\rm (a)]
\item There exists a complete filtration $\CF$ of $V$ for some tower $L$.
\item If additionally the representation $R$ is trace class, the multiplicities are finite.
\end{enumerate}
\end{lemma}

\begin{proof}
(a) A filtration $\CF$ of $(R,V)$, indexed by a sub-tower, is called \e{admissible}, if every $i\in L$ has at least one neighbor $j$ such that $F_i/F_j$ or $F_j/F_i$ respectively, is irreducible.
We apply the Lemma of Zorn to the set of all admissible filtrations $\CF$, where we say that $(\CF,L)\le (\CF',L')$ if $L$ is a subset of $L'$ and the filtration steps of $\CF$ and $\CF'$ agree on $L$.
We get a maximal admissible filtration. We complete $L$ by Dedekind cuts. If $D\subset L$ is a Dedekind cut, i.e., a subset with the property $x<y\in D\Rightarrow x\in D$, then we set $F_D=\ol{\bigcup_{i\in D}F_i}$. If $D\in L$ already, i.e., there exists $\al\in L$ such that $D=\{i\in L: i\le \al\}$, then $F_D=F_\al$, so this filtration extends $\CF$. If $D$ has no neighbor, then it is not in $L$ and $F_D$ is the closure of $\bigcup_{j<D}F_j$ by definition. On the other hand, we have $F_D=\bigcap_{j>D}F_j$, since otherwise there would be an irreducible subquotient between $F_D$ and this intersection, which would contradict the maximality of $\CF$.
Next if $D$ has an upper neighbor, but no lower, then we get $F_D=\ol{\bigcup_{j<D}F_j}$ again by maximality and likewise, we get $F_D=\bigcap_{j>D}F_j$ if $D$ has a lower neighbor, but no upper. This shows that there exists a complete filtration.

(b) Assume the representation to be trace class. By choosing an orthonormal basis which is compatible with the filtration, we see that the trace of $R(f)$ equals the trace on the associated graded representation. This implies finiteness of the multiplicities.
\end{proof}

\section{Admissible representations}
In this section, we assume that  $G$ is a Lie group. Choose a left-invariant metric on $G$ and let $\Delta$ denote the Laplace operator for this metric. We call such a $\Delta$ a \e{group-Laplacian}.
Let $\g_\R$ be the real Lie algebra of $G$ and $\g$ its complexification.
The universal enveloping algebra $U(\g)$ can be identified with the algebra of left-invariant differential operators on $G$, so $\Delta$ can be viewed as an element of $U(\g)$.

By a \e{representation} of $G$ we mean a group homomorphism $R:G\to \GL(V)$ to the group of invertible bicontinuous linear operators on some Banach space $V$ such that the map $G\times V\to V$, $(g,v)\mapsto \pi(g)v$ is continuous.
The space of \e{smooth vectors} $V^\infty$ then is defined as the space of all $v\in V$ such that $G\to V$, $x\mapsto R(x)v$ is infinitely differentiable.
The universal enveloping algebra $U(\g)$ acts on the dense subspace $V^\infty$ of smooth vectors.

\begin{definition}\label{def1.2}
A representation $(R,V)$ of $G$ is called \e{$\Delta$-admissible}, if 
\begin{enumerate}[\rm(a)]
\item there is a dense subset $\Lambda_R\subset\C$, such that for each $\la\in\La_R$ the operator $R(\Delta-\la)^{-1}$ is defined and extends to a continuous operator on the space $V$.
For every $G$-stable closed subspace $U\subset V$ one has $(\Delta-1)^{-1}U\subset U$,

\item for each $\sigma\in\C$ the generalized eigenspace
$$
V(\Delta,\sigma)=\bigcup_{n\in\N}\ker(\Delta-\sigma)^n\subset V^\infty
$$
is finite-dimensional,
\item the set $\Spec_R(\Delta)$ of all $\sigma\in\C$ with $V(\Delta,\sigma)\ne 0$ has no accumulation point in $\C$,
\item every $v\in V$ can be written as absolutely convergent sum
$$
v=\sum_{\sigma\in\Spec(\Delta)}v_\sigma,
$$
each $v_\sigma\in V(\Delta,\sigma)$ is uniquely determined and the projection map $v\mapsto v_\sigma$ is continuous,
\item for every $\sigma_0\in \Spec(\Delta)$  the space
$$
V(\Delta,\sigma_0)'=\ol{\bigoplus_{\sigma\ne\sigma_0}V(\Delta,\sigma)}
$$
satisfies $V=V(\Delta,\sigma_0)\oplus V(\Delta,\sigma_0)'$ and 
the operator $\Delta-\sigma_0$ has a bounded inverse on $V(\Delta,\sigma_0)'$.
\end{enumerate}

The condition (a) needs explaining: We request that there exists a continuous operator $T$ on $V$ which preserves $V^\infty$ as well as every $G$-stable closed subset and satisfies
$$
T R(\Delta-\la)v=R(\Delta-\la)Tv=v
$$
for every $v\in V^\infty$.
We denote this operator by $R(\Delta-\la)^{-1}$.

We find it convenient to leave out the $R$ in the notation, so we occasionally write $(\Delta-\la)$ instead of $R(\Delta-\la)$ and the same for the inverses.
\end{definition}

\begin{lemma}
Let $(R,V)$ be $\Delta$-admissible and let $U\subset V$ be a closed $G$-stable subspace, then $U$ is $\Delta$-admissible.
\end{lemma}

\begin{proof}
The only part of the definition which needs proving, is part (d). More precisely we need to show that if $u\in U$ and $u=\sum_{\sigma}u_\sigma$ is the spectral decomposition in $V$, then $u_\sigma\in U$ for every $\sigma$.
For this let $\sigma_0\in\Spec_R(\Delta)$ and let $\la\in\La_R$ be closer to $\sigma_0$ than any other $\sigma\in\Spec_R(\Delta)$.
Then the operator
$$
T=(\sigma_0-\la)(\Delta-\la)^{-1}
$$
has eigenvalue 1 on $V(\Delta,\sigma_0)$ and eigenvalue of absolute value $<1$ on $V(\Delta,\sigma)$ for every $\sigma_0\ne \sigma\in\Spec_R(\Delta)$.
We write $u=u_{\sigma_0}+u^{\sigma_0}$, where $u^{\sigma_0}=\sum_{\sigma\ne \sigma_0}u_\sigma$.
We first show that $T^nu^{\sigma_0}$ tends to $0$ as $n\to\infty$.
For this note that on the space $V(\Delta,\sigma_0)'$ one has
\begin{align*}
(\Delta-\la)^{-1}&=(\Delta-\la)^{-1}-(\Delta-\sigma_0)^{-1}+(\Delta-\sigma_0)^{-1}\\
&=(\la-\sigma_0)(\Delta-\la)^{-1}(\Delta-\sigma_0)^{-1}+(\Delta-\sigma_0)^{-1}.
\end{align*}
Taking operator norms on both sides and using the triangle inequality we infer that for small values of $|\sigma_0-\la|$ we have
$$
\norm{(\Delta-\la)^{-1}}\le\frac{\norm{(\Delta-\sigma_0)^{-1}}}
{1-(\sigma_0-\la)\norm{(\Delta-\sigma_0)^{-1}}},
$$
where we mean the operator norm on the space $V(\Delta,\sigma_0)'$.
It follows that for $\la$ close enough to $\sigma_0$ the operator norm of $T$ on $V(\Delta,\sigma_0)'$ is $<1$, which implies that $T^nu^{\sigma_0}$ tends to zero.

On $V(\Delta,\sigma_0)$ we write $\Delta=\sigma_0-S$ where $S$ is nilpotent.
So
$$
T=\(1+\sum_{j=1}^{d-1}\frac{S^j}{(\sigma_0-\la)^j}\)
=\(1+N\)
$$
where $N=\sum_{j=1}^{d-1}\frac{S^j}{(\sigma_0-\la)^j}$ is again nilpotent and $d=\dim V(\Delta,\sigma_0)$.
Then on $V(\Delta,\sigma_0)$ we have
$$
T^n=(1+N)^n
=\ \sum_{k=0}^{d-1}\binom nk N^k.
$$
it follows that $T^n\binom n{d-1}^{-1}$ tends to $N^{d-1}$ as $n\to\infty$, 
which implies that $N^{d-1}u_{\sigma_0}$ lies in $U$. 
Next $\(T^n-\binom n{d-1}N^{d-1}\)\binom n{d-2}^{-1}$ 
tends to $N^{d-2}$ which implies that $N^{d-2}u_{\sigma_0}$ lies in $U$. 
We repeat until we reach $N^0u_{\sigma_0}=u_{\sigma_0}\in U$ as claimed.
\end{proof}

\begin{definition}
Let $(R,V)$ be a representation of the locally compact group $G$ and let $\pi$ be an irreducible representation of $G$.
A \e{$\pi$-filtration} in $V$ is a sequence 
$$
F_1'\subset F_1\subset F_2'\subset F_2\subset\dots\subset F_l'\subset F_l
$$
of closed, $G$-stable subspaces such that $F_j/F_j'\cong \pi$ for each $j$.
\end{definition}

\begin{theorem}[Spectral theorem]\label{thm1.6}
Let $(R,V)$ be a $\Delta$-admissible representation of the Lie group $G$.
\begin{enumerate}[\rm (a)]
\item If $V_0\subset V_1$ are closed $G$-stable subspaces, then the sub-quotient $S=V_1/V_0$ is $\Delta$-admissible as well. Each spectral value $\la$ of $S$ is a spectral value of $V$, more precisely, one has
$$
S(\Delta,\la)\cong V_1(\Delta,\la)/V_0(\Delta,\la).
$$
If $m(S,\la)=m(V,\la)$ for all $\la$, then $S=V$.
\item Let $(R,V)$ be $\Delta$-admissible and $\pi$ an irreducible representation of $G$.
Then all maximal $\pi$-filtrations have the same finite length. We call this length $N_{\Ga,\om}(\pi)\in\N_0$ the \e{multiplicity} of $\pi$ in $R$.
\item If $f\in C_c(G)$ is such that the operator $R(f)=\int_Gf(x)R(x)\,dx$ is trace class, then $\pi(f)$ is trace class for every $\pi\in\widetilde G$ with $N_{\Ga,\om}(\pi)>0$ and one has
$$
\tr R(f)=\sum_{\pi\in\widetilde G}N_{\Ga,\om}(\pi)\tr\pi(f).
$$
\item The representation $(R,V)$ is discrete, so there exists a complete filtration on $(R,V)$.
\end{enumerate}
\end{theorem}

\begin{proof}
(a) A submodule is admissible, so it remains to show that a quotient is admissible.
So let $U\subset V$ be a closed $G$-stable subspace. We claim that for $\la_0\in\C$ the map $V(\Delta,\la_0)\to V/U$ induces an isomorphism $ V(\Delta,\la_0)/U(\Delta,\la_0)\cong (V/U)(\Delta,\la_0)$.
The injectivity is clear.
For the surjectivity let $v+U$ be in $(V/U)(\Delta,\la_0)$, then $(\Delta-\la_0)^nv\in U$ for some $n$.
Write $v=\sum_{\la\in\C}v_\la$ as in the definition of admissibility.
We claim that $v-v_{\la_0}$ lies in $U$.
Let $\xi\in\C\sm\Spec_R(\Delta)$.
Write $(\Delta-\la_0)^{n}v=\sum_\la w_\la\in U$,
then each $w_\la$ lies in $U$ and 
$$
(\Delta-\xi)^{-n}\underbrace{(\Delta-{\la_0})^nv}_{\in U}=(\Delta-\xi)^{-n}\sum_\la w_\la
=\sum_{\la}(\Delta-\xi)^{-n}w_\la\in U.
$$
which implies $(\Delta-\xi)^{-n}w_\la\in U$ the uniqueness of the $\la$-expansion.

For $\la\ne \la_0$ we let $\xi$ tend to $\la_0$ and find $(\Delta-\la_0)^{-n}w_\la\in U$.
Next let $(\Delta-\xi)^nv=\sum_\la w_\la^\xi$ and note that $w_\la^\xi$ depends continuously on $\xi$. 
As
$$
v=(\Delta-\xi)^{-n}(\Delta-\xi)^nv=\sum_\la(\Delta-\xi)^{-n}w_\la^\xi,
$$
we deduce $v_\la=(\Delta-\xi)^{-n}w_\la^\xi$ by uniqueness.
For $\la\ne\la_0$ we let $\xi$ tend to $\la_0$ and we can deduce $v_\la=(\Delta-\la_0)^{-n}w_\la\in U$.
This implies $v-v_{\la_0}\in U$ as claimed. The rest of part (a) is clear.

For (b),(c) and (d) we argue that for an admissible representation the property (d) implies (b) and (c).
To see that (d) implies (b) we consider a maximal $\pi$-filtration
$$
F_1'\subset F_1\subset F_2'\subset F_2\subset\dots\subset F_l'\subset F_l
$$
and 
a complete $G$-stable $L$-filtration $(S_j)_{j\in L}$ with irreducible quotients.
We claim that there must exist indices $\nu_1<\nu_1'<\dots<\nu_l<\nu_l'$ in $L$ such that  $S_{\nu_i}/S_{\nu_i'}\cong\pi$ and $S_{\nu_i'}/S_{\nu_{i-1}}$ has no $\pi$-sub-quotient, so that $l$ equals the number of $\pi$-sub-quotients within the given $L$-filtration and this independent of the chosen maximal $\pi$-filtration.
If the $L$-filtration is finite, this is the classical Jordan-Hölder Theorem.
We reduce the present case to a finite filtration as follows: We choose a $\la\in\Spec_\pi(\Delta)$.
Then $(S_j^\la=S_j\cap V(\Delta,\la))_j$ is a filtration of this finite dimensional space.
There must exist two neighboring indices $i_1<j_1$ such that $S_{i_1}^\la=0$ and $S_{j_1}^\la\ne 0$.
Repeating we find indices $i_1<j_1<i_2<j_2<\dots<i_k<j_k$ such that $i_\nu$ and $j_\nu$ are neighbored for each $\nu$ and $S_{j_\nu}^\la=S_{i_{\nu+1}}$ always holds, which implies that $S_{i_{n+1}}/S_{j_\nu}$ has no $\pi$-sub-quotient.
Further $S_{i_1}$ and $V/S_{j_k}$ both have no $\pi$-sub-quotient.
Now one can ignore the $\nu$ with $S_{j_\nu}/S_{i_\nu}\ncong \pi$ and assume that all quotients are $\cong\pi$. From here on the classical proof of the Jordan-Hölder Theorem applies to show that $k=l$.
After that, once we know that $N_{\Ga,\om}(\pi)$ equals the number of $\pi$-sub-quotients in the given $L$-filtration, part (c) also follows.

So it remains to show (d).
Let $\la\in\Spec_R(\Delta)$. By Zorn's lemma there exists a maximal $G$-stable subspace $V_0$ such that $V_0\cap V(\Delta,\la)=0$.
Then its closure $\ol V_0$ is admissible and thus satisfies the same claim, i.e., $\ol V_0\cap V(\Delta,\la)=0$, so by maximality, $V_0$ is closed.
Let $v\in V(\Delta,\la)$ and let $S(v)$ denote the closure of the span of $V_0+R(G)v$.
Among all spaces $S(v)$ as $v$ varies in $V(\Delta,\la)\sm\{ 0\}$, there is a minimal one $V_1$.
Then $V_1/V_0$ is irreducible.
\end{proof}

\section{The spectral theorem}
Let $G$ be a locally compact group and let $\Ga\subset G$ be a cocompact lattice.
This means that $\Ga$ is a discrete subgroup such that the quotient $\Ga\bs G$ is compact.
Let $\om:\Ga\to\GL(V)$ be a group homomorphism, where $V=V_\om$ is a finite-dimensional complex vector space.
Let $E=E_\om=\Ga\bs (G\times V_\om)$, where $\Ga$ acts diagonally.
The projection onto the first factor makes $E$ a  vector bundle over $\Ga\bs G$.
The space $\Ga(E)$ of continuous sections can be identified with the space $C(\Ga\bs G,\om)$  of all continuous functions $f:G\to V_\om$ such that $f(\ga x)=\om(\ga)f(x)$ for all $\ga\in\Ga$.
Choose a hermitian metric on $E$ to define the space $L^2(E)$ of $L^2$-sections.
This space can be identified with the space $L^2(\Ga\bs G,\om)$ of all measurable functions $f:G\to V_\om$ with $f(\ga x)=\om(\ga)f(x)$ and $\int_F\sp{f(x),f(x)}_x\,dx<\infty$, where $F\subset G$ is a compact fundamental domain for $\Ga\bs G$.
The group $G$ acts by right translations on the Hilbert space $L^2(\Ga\bs G,\om)$.
This representation is continuous but in general not unitary.
Let $R=R_{\Ga,\om}$ denote the right regular representation of $G$ on the Hilbert space $H=L^2(\Ga\bs G,\om)$.

\begin{theorem}
Let $G$ be a Lie group and $\Ga\subset G$ a cocompact lattice.
Fix a group-Laplacian $\Delta$.
Then the representation $(R,V)$ with $V=L^2(\Ga\bs G,\om)$ is $\Delta$-admissible.
In particular, there exists a complete filtration for $(R,V)$, the multiplicities $m(\pi)$ of which are finite and given by the maximal lengths of $\pi$-filtrations.
\end{theorem}

\begin{proof}
The element $\Delta\in U(\g)$ acts on $C^\infty(\Ga\bs G,\om)$ as a differential operator of order two whose principal symbol equals the square of the norm given by the Riemannian metric, such operators are called \e{generalized Laplacians} in \cite{BGV}.
By \cite{Shubin}*{Theorems 8.4 and 9.3} and \cite{Markus}*{Theorem4.3} it follows that $\Delta$ has discrete spectrum in $L^2(\Ga\bs G,\om)$, i.e., there exists a sequence $\la_j$ of complex numbers which do not accumulate in $\C$ such that the space $\bigoplus_{j=1}^\infty H(\Delta,\la_j)$ is dense in $H=L^2(\Ga\bs G,\om)$.
Each $v\in H$ can uniquely be written as convergent $\sum_j u_j$ with $u_j\in H(\Delta,\la_j)$.

One sets $\La_R$ equal to $\C\sm\{\la_j:j\in\N\}$.
Then for given $\la\in\La_R$ the space $H(\Delta,\la)$ which lies in  $C^\infty(\Ga\bs G,\om)$, is finite-dimensional.
The only tricky point is to show that for a given closed $G$-stable subspace $U\subset H$ one has $(\Delta-\xi)^{-1}U\subset U$ for $\xi\in\Lambda_R$.
For this note that $(\Delta-\xi)^{-1}=f(\sqrt\Delta)$ with $f(x)=(x^2-\xi)^{-1}$.
The Fourier transform of $f$ is $\hat f(x)=\frac{e^{i|x|\sqrt\xi}}{2\sqrt\xi}$, where $\sqrt\xi$ denote the unique complex number $\al$ with $\Im(\al)>0$ and $\al^2=\xi$.
Let $\chi_0$ be a smooth function on $\R$ with $0\le\chi\le 1$, $\chi(t)=1$ for $t\le 0$ and $\chi_1(t)=0$ for $t\ge 1$.
For $T>0$ set $\chi_T(t)=\chi_0(t-T)$ and let $f_T$ be defined by $\hat f_T(x)=\chi_T(|x|)\hat f(x)$.
Then $\hat f_T(x)$ has compact support and by \cite{CGT} it follows that the operator $f_T(\sqrt\Delta)$ has finite propagation speed. We can view this operator of $G$ or on $\Ga\bs G$.
The connection between the two is as follows: On $G$ this operator is invariant under left translation by elements of $G$, hence it is given by right convolution with a function, which, by finite propagation speed, has compact support.
This function is continuous on $G$ and smooth on the set $G\sm\{ 1\}$.
We denote it by $x\mapsto f_T(\sqrt\Delta)(x)$.
Then on $\Ga\bs G$ the operator $f_T(\sqrt\Delta)$ has continuous kernel $k(x,y)=\sum_{\ga\in\Ga}f_T(\sqrt\Delta)(x^{-1}\ga y)$, the sum being locally finite.
For $\phi\in L^2(E)$ one has $R(f_T(\sqrt\Delta))\phi(x)=\int_Gf_T(\sqrt\Delta)(y)\phi(xy)\,dy$ and approximating this integral by Riemann sums, one sees that $R(f_T(\sqrt\Delta))\phi$ lies in $U$ if $\phi\in U$.
It therefore suffices to show that $R(f_T(\sqrt\Delta))\phi$ converges to $R(\Delta-\xi)^{-1}\phi$ as $T\to\infty$.
On the compact manifold $\Ga\bs G$ this follows if we show that the kernel of the former converges uniformly to the kernel of the latter, which is a consequence of Theorem 1.4 of \cite{CGT}.
\end{proof}

\begin{theorem}[Trace formula]
Let $G$ be a locally compact group and let $\Ga\subset G$ be a cocompact lattice.
Let $(\om,V_\om)$ be a representation of the discrete group $\Ga$ on a finite-dimensional complex vector space $V_\om$ and define the Hilbert space $H=L^2(\Ga\bs G,\om)$ as above.
Then for each $f\in C^\infty_c(G)$ the operator $R(f)$ is trace class and its trace equals either side of the equation
$$
\sum_{\pi\in\widetilde G}N_{\Ga,\om}(\pi)\tr\pi(f)=\sum_{[\ga]}\vol(\Ga_\ga\bs G_\ga)\CO_\ga(f)\tr\om(\ga),
$$
where $N_{\Ga,\om}(\pi)$ denotes the maximal length of a $\pi$-filtration in $H$, the sum on the right runs over all conjugacy classes $[\ga]$ in $\Ga$, the groups $G_\ga$ and $\Ga_\ga$ are the centralizers of $\ga$ in $G$ and $\Ga$ and $\CO_\ga$ denotes the \e{orbital integral}
$$
\CO_\ga(f)=\int_{G_\ga\bs G}f(x^{-1}\ga x)\,dx.
$$
The left hand side of the formula is also called the \e{spectral side} and the right hand side is the \e{geometric side}.
\end{theorem}

\begin{proof}
First assume that $G$ is a Lie group.
By the Theorem of Dixmier and Malliavin \cite{DM}, every $f\in C_c^\infty(G)$ is a finite sum of convolution products $g*h$ with $g,h\in C_c^\infty(G)$.
If $f=g*h$ then $R(f)=R(g)R(h)$.
Now the same calculus as in the unitary case \cite{HA2}*{Chapter 9} implies that $R(f)$ is an integral operator with smooth kernel $k(x,y)=\sum_{\ga\in\Ga}f(x^{-1}\ga y)\om(\ga)$, so by \cite{HA2}*{Proposition 9.3.1} it is trace class and its trace equals 
$\int_{\Ga\bs G}\tr k(x,x)\,dx$, which with the same computation as in the proof of \cite{HA2}*{Theorem 9.3.2} is seen to be equal to 
$$
\sum_{[\ga]}\vol(\Ga_\ga\bs G_\ga)\CO_\ga(f)\tr\om(\ga).
$$
We this get the geometric side of the trace formula.
The spectral side is obtained from Theorem \ref{thm1.6}.

To finish the proof, we generalize the trace formula to arbitrary locally compact groups.
So assume now that $G$ is the projective limit of its Lie quotients,
$$
G=\lim_{\substack{\leftarrow\\ N}}G/N.
$$
A given $f\in C_c^\infty(G)$ will factorize over some Lie quotient $G/N$.
We can assume the compact group $N$ chosen so small that $N\cap\Ga=\{1\}$.
Then $\Ga$ induces a cocompact lattice in $G/N$ and the trace formula for this group implies the trace formula for the given $f$.

Finally, assume that trace formula holds for an open subgroup $H$ of $G$, then $\Ga\cap H$ is a cocompact lattice in $H$ and the trace formula for  $H$ implies the trace formula for $G$.
\end{proof}

\section{Semisimple Lie groups}
In the case of a semisimple group $G$ we here prove a slightly stronger spectral theorem which says that the right regular representation on $L^2(E)$ is a direct sum of representations of finite length.

\begin{definition}
A representation $(R,V)$ of a locally compact group has \e{finite length}, if there exists a filtration
$$
0=F_0\subset\dots\subset F_k=V
$$
of closed $G$-stable subspaces such that $F_j/F_{j-1}$ is irreducible for each $j$.
The classical Jordan-Hölder Theorem says that then the irreducible quotients $F_j/F_{j-1}$ are uniquely determined by $V$ up to order.
\end{definition}

\begin{definition}
We say that a representation $(R,V)$ of a locally compact group $G$ is a \e{Jordan-Hölder representation}, if it is a direct sum of finite length representations.
More precisely, we insist that there are closed $G$-stable subspaces $V_i$, $i\in I$ such that the direct sum $\bigoplus_{i\in I}V_i$ is dense in $V$.
\end{definition}

Let $G$ be a semisimple Lie group with finite  center and let $K$ be a maximal compact subgroup. 
Let $\Ga\subset G$ be a cocompact lattice and let $(\chi,V_\chi)$ be a finite dimensional complex representation of $\Ga$.
Then $\chi$ defines a vector bundle $E=E_\chi$ over $\Ga\bs G$.
The smooth sections can be described as
$$
\Ga^\infty(E)\cong (C^\infty(G)\otimes V_\chi)^\Ga.
$$
The choice of a hermitian metric on $E$ allows the definition of the Hilbert space $L^2(E)$ of square integrable sections.
We equip $\Ga^\infty(E)$ with the topology of $L^2(E)$.

Let $V_\fin$ be the space of all sections in $\Ga^\infty(E)$ which are $K$-finite as well as $\z$-finite, where $\z$ is the center of the universal covering algebra $U(\g_\C)$ of the complexified Lie algebra $\g_\C$ of $G$.

\begin{theorem}
The $(\g,K)$-module $V_\fin$ is dense in $\Ga^\infty(E)$ as well as in $L^2(E)$.
The $G$-representations on $\Ga^\infty(E)$ and on $L^2(E)$ are Jordan-Hölder representations.
\end{theorem}

\begin{proof}
For every $(\tau,V_\tau)\in\hat K$ the Casimir element $C\in \z$ acts on the $\tau$-isotype 
$$
\Ga^\infty(E)(\tau)\cong V_\tau\otimes \Hom_K(V_\tau,\Ga^\infty(E)),
$$
as it acts on 
\begin{eqnarray*}
\Hom_K(V_\tau,\Ga^\infty(E)) &\cong&
\(\Ga^\infty(E)\otimes V_\tau\)^K\\
&\cong& \( C^\infty(G)\otimes V_\chi\otimes V_\tau\)^{\Ga\times K}\\
&\cong& \Ga^\infty(E_{\chi,\tau}),
\end{eqnarray*}
where $E_{\chi,\tau}$ is the vector bundle over $\Ga\bs G/K$ defined by $\chi\times\tau$.
On $\Ga^\infty(E_{\chi,\tau})$ the Casimir $C$ induces an operator which has the same principal symbol as the Laplacian for any given metric.
Hence (\cite{Shubin}, Theorems 8.4 and 9.3) the operator $C$ has discrete spectrum on $L^2(E_{\chi,\tau})$ consisting of eigenvalues of finite multiplicity.

Let $\la\in\C$ be an eigenvalue and let $\Ga^\infty(E_{\chi,\tau})(\la)$ be the corresponding finite dimensional generalized eigenspace.
The image $V_{\tau,\la}$ of $\Ga^\infty(E_{\chi,\tau})(\la)$ in $\Ga^\infty(E)$ is $\z$-stable and $K$-stable.
Hence the generated $(\g,K)$-module $U(\g)V_{\tau,\la}$ 
is in $V_\fin$ and by Corollary 3.4.7 of \cite{Wall} 
it is admissible and as it is finitely generated, it 
is a Harish-Chandra module, so by Corollary 10.42 of 
\cite{Knapp} it has a finite composition series:
$$
U(\g)V_{\tau,\la}= F_k\supset F_{k-1}\supset\dots\supset F_0=0
$$
with irreducible quotients $F_{j+1}/F_j$.
We repeat this argument with a different $K$-type $\tau'$ not occurring in $U(\g)V_{\tau,\la}$ if it exists. Otherwise, we repeat it with a different eigenvalue $\la$ to get the claim.
\end{proof}

{\small Mathematisches Institut\\
Auf der Morgenstelle 10\\
72076 T\"ubingen\\
Germany\\
\tt deitmar@uni-tuebingen.de}

\today

\end{document}